\documentclass[a4paper,11pt]{article}
\usepackage{amsmath}
\usepackage{latexsym,a4wide}
\usepackage{amssymb}

\input xypic
\input xy
\xyoption{all}

\newtheorem{example}{Example}[section]
\newtheorem{definition}[example]{Definition}
\newtheorem{proposition}[example]{Proposition}
\newtheorem{theorem}[example]{Theorem}
\newtheorem{corollary}[example]{Corollary}
\newtheorem{lemma}[example]{Lemma}

\newtheorem{remark}[example]{Remark}
\newenvironment{proof}{\noindent \textbf{Proof:} }{ $\Box$ \mbox{} }

\begin{document}

\author{U. Ege Arslan \ and \"{O}. G\"urmen }
\title{CHANGE OF BASE FOR COMMUTATIVE ALGEBRAS}

\maketitle

\begin{abstract}
In this paper we examine on a pair of adjoint functors $(\phi ^{\ast
},\phi _{\ast })$ for a subcategory of the category of crossed modules
over commutative algebras where $\phi ^{\ast }:\mathbf{XMod}$\textbf{/}$%
Q\rightarrow $ $\mathbf{XMod/}P$, pullback, which enables us to move
from crossed $Q$-modules to crossed $P$-modules by an algebra
morphism $\phi :P\rightarrow Q$ and $\phi _{\ast }:\mathbf{XMod}$\textbf{/}$P\rightarrow $ $\mathbf{XMod/}Q$, induced. We note that this adjoint functor pair $(\phi ^{\ast
},\phi _{\ast })$ makes $p:\mathbf{XMod}\rightarrow $ $k$\textbf{-Alg }into a bifibred category over  $k$\textbf{-Alg}, the category of commutative algebras, where $p$ is given by $p(C,R,\partial )=R.$
 Also, some examples and results on
induced crossed modules are given.
\end{abstract}

\begin{tabular}{l}
\textbf{A. M. S. Classification}:  18A40,18A30. 
\end{tabular}
\section*{Introduction}

Let $S$ be a ring. Then there is a category {\bf Mod}/$S$ of modules over $%
S. $ If $\phi :S\rightarrow R$ is a ring homomorphism, then there is
a functor $\phi ^{\ast }$ from ${\bf Mod}/R$ to ${\bf Mod}/S$ where
$S$ acts on an $R$-module via $\phi .$ This functor has a left
adjoint $\phi _{\ast }$ to $\phi ^{\ast }$ giving the well known
{\it induced module} via tensor product. This construction is
known as a ``{\it change of base}'' in a general module theory
setting. Brown and Higgins \cite{[bh]} generalised that to higher
dimension for the group theoretical case, that is, a morphism $\phi
:P\rightarrow Q$ of groups determines a pullback functor $\phi ^{\ast }:{\bf %
XMod}${\bf /}$Q\rightarrow $ ${\bf XMod/}P,$ where ${\bf XMod/}Q$
denotes a subcategory of the category of crossed modules ${\bf XMod}$
and has as objects those crossed modules with a fixed group $Q$ as the ``base''. The left adjoint $\phi _{\ast }$ to pullback gives the {\it
induced crossed modules}. This is also given by pushouts of crossed
modules.

In this work we will consider the appropriate analogue of that in
the theory of crossed modules in commutative algebras. Although this
construction has already been worked by Porter \cite{[p1]} and
Shammu \cite{[s]}, we will reconsider and develop that in the light
of the works of Brown and Wensley \cite{[bw1],[bw2]}. The purpose of
this paper is to give some new examples and results on crossed
modules induced by a morphism of algebras $\phi :S\rightarrow R$ in
the case when $\phi $ is the inclusion of an ideal. In
the applications to commutative algebras, the induced crossed modules play
an important role since the free crossed modules which are related to Koszul
complexes given by Porter \cite{[p2]} are the special case of induced
crossed modules. In \cite{[p2]}, any finitely generated free crossed module $%
C\rightarrow R$ of commutative algebras was shown to have $C\cong
R^{n}/d(\Lambda ^{2}R^{n})$, i.e. the $2$nd Koszul complex term module the $%
2 $-boundaries where $d:\Lambda ^{2}R^{n}\rightarrow R^{n}$ the Koszul
differential. So, we think that the induced crossed modules of commutative
algebras give useful information on Koszul-like constructions.

Conventions: Throughout this paper $k$ is a fixed commutative ring, $R$ is a $k$%
-algebra with identity. All $k$-algebras will be assumed
commutative and associative but there will not be requiring
algebras to have unit elements unless stated otherwise.

Acknowledgements: The authors wishes to thank Z.Arvasi for his helpful comments.

\section{ Crossed Modules of Commutative Algebras}

{\em A crossed module} of algebras, $(C,R,\partial ),$ consists of an $R$%
-algebra $C$ and a $k$-algebra $R$ with an action of $R$ on $C,$ $%
(r,c)\mapsto r\cdot c$ for $c\in C,r\in R,$ and an $R$-algebra morphism $%
\partial :C\rightarrow R$ satisfying the following condition for all $%
c,c^{\prime }\in C$
\[
\partial (c)\cdot c^{\prime }=cc^{\prime }.
\]
This condition is called the{\em \ Peiffer identity}. We call $R,$
the base algebra and $C,$ the top algebra. When we wish to
emphasise the base algebra $R,$ we call $(C,R,\partial ),$ a
crossed $R$-module.

\smallskip {\it A morphism of crossed modules} from $(C,R,\partial )$ to $%
(C^{\prime },R^{\prime },\partial ^{\prime })$ is a pair $(f,\phi )$ of $k$%
-algebra morphisms $f:C\longrightarrow C^{\prime },\phi
:R\longrightarrow R^{\prime }$ such that

\[
\text{(i) }\partial ^{\prime }f=\phi \partial \qquad \text{and}\qquad \text{%
(ii) \ }f(r\cdot c)=\phi (r)\cdot f(c)
\]
for all $c\in C,r\in R.$ Thus one can obtain the category {\bf
XMod} of crossed modules of algebras. In the case of a morphism
$(f,\phi )$ between crossed modules with the same base $R$, say,
where $\phi $ is the identity
on $R$,

$$
\xymatrix@R=20pt@C=20pt{
  C \ar[rr]^{f} \ar[dr]_{\partial}
                &  &    C' \ar[dl]^{\partial'}    \\
                & R                 }$$
then we say that $f$ is a morphism of crossed $R$%
-modules. This gives a subcategory {\bf XMod/}R of {\bf XMod.} \

\subsection{Examples of Crossed Modules}

(i) Any ideal, $I$, in $R$ gives an inclusion map
$I\longrightarrow R,$ which is a crossed module then we will say
$\left( I,R,i\right) $ is \ an ideal pair. In this case, of
course, $R$ acts on $I$ by multiplication and the inclusion
homomorphism $i$ makes $\left( I,R,i\right) $ into a crossed
module, an \textquotedblleft inclusion crossed
modules\textquotedblright . Conversely,

\begin{lemma}
If $(C,R,\partial )$ is a crossed module, $\partial (C)$ is an
ideal of $R.$ $\Box $
\end{lemma}

(ii) Any $R$-module $M$ can be considered as an $R$-algebra with
zero multiplication and hence the zero morphism $0:M\rightarrow R$
sending everything in $M$ to the zero element of $R$ is a crossed
module. Again conversely:

\begin{lemma}
If $(C,R,\partial )$ is a crossed module, \text{Ker}$\partial $ is an
ideal in $C$ and inherits a natural $R$-module structure from
$R$-action on $C.$ Moreover, $\partial (C)$ acts trivially on
\text{Ker}$\partial ,$ hence \text{Ker}$\partial $ has a natural $R/\partial (C)$-module structure. $\Box $
\end{lemma}

As these two examples suggest, general crossed modules lie between
the two extremes of ideal and modules. Both aspects are
important.

(iii) In the category of algebras, the appropriate replacement for
automorphism \ groups is the bimultiplication algebra Bim$\left( R\right) $
defined by Mac Lane \cite{[m]}. (see also \cite{[dl],[lue]}). Let $R$ be an
associative (not necessarily unitary or commutative) $k$-algebra. Bim$(R)$
consists of pairs $(\gamma ,\delta )$ of $R$-linear mappings from $R$ to $R$
such that
\begin{equation*}
\gamma (rr^{\prime })=\gamma (r)\cdot r^{\prime }
\end{equation*}%
\begin{equation*}
\delta (rr^{\prime })=r\cdot \delta \left( r^{\prime }\right)
\end{equation*}%
and
\begin{equation*}
r\cdot \gamma \left( r^{\prime }\right) =\delta (r)\cdot r^{\prime }.
\end{equation*}

If $R$ is a commutative algebra and $Ann\left( R\right) =0$ or $R^{2}=R,$
then, since%
\begin{equation*}
\begin{array}{lll}
x\cdot \delta (r) & = & \delta (r)\cdot x=r\cdot \gamma (x)=\gamma (x)\cdot r
\\
& = & \gamma (xr)=\gamma (rx)=\gamma (r)\cdot x=x\cdot \gamma (r)%
\end{array}%
\end{equation*}%
for every $x\in R,$ we get $\gamma =\delta .$ Thus Bim$(R)$ may be
identified with the $k$-algebra $\mathcal{M}(R)$ of multipliers of $R.$
Recall that a multiplier of $R$ is a linear mapping $\lambda
:R\longrightarrow R$ such that for all $r,r^{\prime }\in R$%
\begin{equation*}
\lambda (rr^{\prime })=\lambda (r)r^{\prime }.
\end{equation*}%
Also $\mathcal{M}(R)$ is commutative as%
\begin{equation*}
\lambda ^{\prime }\lambda (xr)=\lambda ^{\prime }(\lambda (x)r)=\lambda
(x)\lambda ^{\prime }(r)=\lambda ^{\prime }(r)\lambda (x)=\lambda \lambda
^{\prime }(rx)=\lambda \lambda ^{\prime }(xr)
\end{equation*}%
for any $x\in R.$ Thus $\mathcal{M}(R)$ is the set of all multipliers $%
\lambda $ such that $\lambda \gamma =\gamma \lambda $  for every multiplier $%
\gamma .$

So automorphism crosed module corresponds to the multiplication crossed
module $\left( R,M\left( R\right) ,\mu \right) $ where $\mu :R\rightarrow
M\left( R\right) $ is defined by $\mu \left( r\right) =\lambda _{r}$ with $%
\lambda _{r}\left( r^{\prime }\right) =rr^{\prime }$ for all $r,r^{\prime
}\in R$ and the action is given by $\lambda \cdot r=\lambda \left( r\right) $
(See [2] for details).
\subsection{Free Crossed Modules}\label{s5}

Let $\left( C,R,\partial \right) $ be a crossed module, let $Y$ be
a set and let $\upsilon :Y\rightarrow C$ be a function, then
$\left( C,R,\partial \right) $ is said to be a free crossed module
with basis $\upsilon $ or alternatively, on the function $\partial
\upsilon :Y\rightarrow R$ if for
any crossed $R$-module $\left( A,R,\delta \right) $ and a function $%
w:Y\rightarrow A$ such that $\delta w=\partial \upsilon $, there
is a unique morphism
\[
\phi :\left( C,R,\partial \right) \longrightarrow \left(
A,R,\delta \right)
\]
such that the diagram

$$
\xymatrix@R=20pt@C=20pt{
  C \ar[rr]^{\phi} \ar[dr]_{\partial}
                &  &    A \ar[dl]^{\delta}    \\
                & R                 }$$
is commutative.

For our purpose, an important standard construction of free crossed $R$%
-modules is as follows:

Suppose given $f:Y\rightarrow R$. Let $E=R^{+}[Y],$ the positively
graded part of the polynomial ring on $Y.$ $f$ induces a morphism
of $R$-algebras,
\[
\theta :E\rightarrow R
\]
defined on generators by
\[
\theta \left( y\right) =f\left( y\right) .
\]
We define an ideal $P$ in $E$ (sometimes called by analogy with
the group theoretical case, the Peiffer ideal relative to $f$)
generated by the elements
\[
\left\{ pq-\theta \left( p\right) q:\ p,q\in E\right\}
\]
clearly $\theta \left( P\right) =0,$ so putting $C=E/P$, one
obtains an induced morphism
\[
\delta :C\rightarrow R
\]
which is the required free crossed $R$-module on $f$ (cf. \cite{[p2]}).

This construction will be seen later as a special case of an
induced crossed module.
\section{Fibrations and Cofibrations Categories}\label{s1}
The notion of fibration of categories is intended to give a general
background to constructions analogous to pullback by a morphism. It
seems to be a very useful notion for dealing with hierarchical
structures. A functor which forgets the top level of structure is
often usefully seen as a fibration or cofibration of categories.

We recall from \cite{[bs]} the definition of fibration and cofibration of categories.
\begin{definition} Let $ \Phi :\mathbf{X}\to \mathbf{B} $ be a functor. A morphism $ \varphi :Y\to
X $ in $ \mathbf{X} $ over $ u := \Phi (\varphi) $ is called
cartesian if and only if for all $ v: K\to J $ in $ \mathbf{B} $ and
$ \theta:Z\to X $ with $ \Phi(\theta)=uv $ there is a unique
morphism $ \psi: Z\to Y $ with $ \Phi(\psi)=v $ and $
\theta=\varphi\psi . $

This is illustrated by the following diagram:
$$
\xymatrix{
& Z \ar@/^1pc/[rr]^{\theta}\ar@{.>}[r]_{\psi}  & Y \ar[r]_{\varphi} & X & &\ar[d]_{\Phi}\\
& K \ar@/_1pc/[rr]_{uv}\ar[r]_{v}  & J \ar[r]_{u} & I & & }
$$
It is straightforward to check that cartesian morphisms are closed
under composition, and that $ \varphi $ is an isomorphism if and
only if $ \varphi $ is a cartesian morphism over an isomorphism.

A morphism $ \alpha :Z\to Y  $ is called vertical (with respect to $
\Phi $) if and only if $ \Phi(\alpha) $  is an identity morphism in
$ \mathbf{B}. $ In particular, for $ I\in \mathbf{B}$ we write
$\mathbf{X}$/I , called the fibre over $ I $, for the subcategory of
$ \mathbf{X} $ consisting of those morphisms $ \alpha $ with $
\Phi(\alpha)=id_I . $
\end{definition}

\begin{definition}
The functor $ \Phi :\mathbf{X}\to \mathbf{B} $ is a fibration or
category fibred over $ \mathbf{B} $ if and only if for all $ u:
J\to I $ in $ \mathbf{B} $ and $ X\in\mathbf{X}$/I there is a
cartesian morphism $\varphi: Y\to X $ over $ u: $ such a $ \varphi $ is
called a cartesian lifting of  $ X $ along $ u. $
\end{definition}

In other words, in a category fibred over $\mathbf{B}$, $ \Phi :\mathbf{X}\to \mathbf{B} $, we can pull back objects of $\mathbf{X}$ along any arrow of $\mathbf{B}$.

\begin{definition} Let $ \Phi :\mathbf{X}\to \mathbf{B} $ be a functor. A morphism $ \varphi :Z\to
Y $ in $ \mathbf{X} $ over $ v := \Phi (\psi) $ is called
cocartesian if and only if for all $ u: J\to I $ in $ \mathbf{B} $
and $ \theta:Z\to X $ with $ \Phi(\theta)=uv $ there is a unique
morphism $ \varphi: Y\to X $ with $ \Phi(\varphi)=u $ and $
\theta=\varphi\psi . $ This is illustrated by the following diagram:
$$
\xymatrix{
& Z \ar@/^1pc/[rr]^{\theta}\ar[r]_{\psi}  & Y \ar@{.>}[r]_{\varphi} & X & &\ar[d]_{\Phi}\\
& K \ar@/_1pc/[rr]_{uv}\ar[r]_{v}  & J \ar[r]_{u} & I & & }
$$
It is straightforward to check that cocartesian morphisms are closed
under composition, and that $ \psi $ is an isomorphism if and only
if $ \psi $ is a cocartesian morphism over an isomorphism.

The functor $ \Phi :\mathbf{X}\to \mathbf{B} $ is a cofibration or
category cofibred over $ \mathbf{B} $ if and only if for all $ v:
K\to J $ in $ \mathbf{B} $ and $ Z\in\mathbf{X}$/K there is a
cartesian morphism $ \psi: Z\to Z' $ over $ v: $ such a $ \psi $ is
called a cocartesian lifting of  $ Z $ along $ v. $

The cocartesian liftings of $ Z\in\mathbf{X}$/K along $ v: K\to J $
are also unique up to vertical isomorphism.
\end{definition}

It is interesting to get a characterisation of the cofibration property for a functor that already is a fibration. The following is a useful weakening of the condition for cocartesian in the case of a fibration of categories.
\begin{proposition}
Let $ \Phi :\mathbf{X}\to \mathbf{B} $ be a fibration of categories. Then $ \psi:Z\to Y $ in $\mathbf{X}$ over $ v:K\to J $ in $\mathbf{B}$ is cocartesian if only if for all $\theta ^{\prime }:Z\rightarrow X^{\prime }$ over $v$ there is a unique
morphism $\psi ^{\prime }:Y\rightarrow X^{\prime }$ in \textbf{X}/J with
$\theta ^{\prime }=\psi ^{\prime }\psi .$
\end{proposition}
The following Proposition allows us to prove that a fibration is also a
cofibration by constructing the adjoints $u_{\ast }$ of $u^{\ast \text{ }}$%
for every $u.$
\begin{proposition}
Let $\Phi :\mathbf{X}\rightarrow \mathbf{B}$ be a fibration of categories.
Let $u:J\rightarrow I$ have reindexing functor $u^{\ast \text{ }%
}:X/I\rightarrow X/J.$ Then the following are equivalent:

(i) For all $Y\in X/J,$ there is a morphism $u_{Y}:Y\rightarrow u_{\ast }Y$
which is cocartesian over $u$;

(ii) there is a functor $u_{\ast }:X/J\rightarrow X/I$ which is left
adjoint to $u^{\ast}.$
\end{proposition}
\section{Pullback Crossed Modules}\label{s3}

Within the theory of modules and more generally of Abelian
categories, there is \ a very important set of results known as
Morita Theory describing between categories of modules. The idea
is that let $\phi :S\rightarrow R$
be a ring homomorphism and let $M$ be a $R$-module, then we can obtain $S$%
-module $\phi ^{\ast }\left( M\right) $ by means of $\phi $ for
which the action is given by $s\cdot m=\phi (s)m,$ for $s\in
S,m\in M.$ Then there is a functor
\[
\phi ^{\ast }:{\bf Mod}/R\longrightarrow {\bf Mod}/S\,.
\]
This functor has a left adjoint
\[
\phi _{\ast }:{\bf Mod}/S\longrightarrow {\bf Mod}/R\,.
\]
Then each $S$-module $N$ defines a $R$-module $\phi _{\ast
}(N)=R\otimes _{S}N.$ This construction is also known as a
``change of base''\ in a module theory. In this section we will
see the corresponding idea with crossed modules. We call these
structures the pullback crossed module and the induced crossed module,
respectively. These functors had already been done by Porter,
\cite{[p1]}, under different names. Also Shammu \cite{[s]}, had
considered in his thesis for non-commutative case. But we will
deeply analyse these constructions by using the work of
Brown-Wensley and Brown, Higgins and Sivera \cite{[bhs],[bw1],[bw2]}. Similar results are known for crossed
modules of groups \cite{[bh],[bw1]}, and Lie algebras \cite{[cl]}.

We will define a pullback crossed module which is due to
\cite{[bhs]}.
\begin{definition}
Given a crossed module $\partial :C\rightarrow R$ and a morphism of $k$%
-algebras $\phi :S\rightarrow R,$ the {\it pullback crossed
module} can be given by

(i) a crossed $S$-module $\phi ^{\ast }\left( C,R,\partial \right)
=(\partial ^{\ast }:\phi ^{\ast }(C)\rightarrow S)$

(ii) given
\[
\left( f,\phi \right) :\left( B,S,\mu \right) \longrightarrow
\left( C,R,\partial \right)
\]
crossed module morphism, then there is a unique $\left( f^{\ast
},id_{S}\right) $ crossed $S$-module morphism that commutes the
following diagram:

$$\xymatrix@R=40pt@C=40pt{
                &   (B,S,\mu)\ar@{.>}[dl]_{(f^{\ast},id_{S})} \ar[d]^{(f,\phi)}     \\
  (\phi^{\ast}(C),S,\partial^{\ast})  \ar[r]_{(\phi',\phi)} & (C,R,\partial)             }
$$
or more simply as
$$\xymatrix@R=20pt@C=20pt{
  B \ar[dd]_{\mu} \ar[rr]^{f}\ar@{.>}[dr]_{f^{\ast}}
              &  & C \ar[dd]^{\partial}  \\
&\phi^{\ast}(C)\ar[ur]_{\phi'}\ar[dl]^{\partial^{\ast}}&
 \\ S  \ar[rr]_{\phi}
               & & R             }$$
\end{definition}

\subsection{Construction of Pullback Crossed Module}\label{s2}

Let $(C,R,\partial )$ be a crossed $R$-module and let $\phi
:S\longrightarrow R$ be a morphism of $k$-algebras. Then $A=\{(c,s)\mid \phi
\left( s\right) =\partial \left( c\right) ,\ s\in S,c\in C\}$ has the
structure of a $S$-algebra by

\[
s\cdot {(c,s^{\prime })}=(\phi \left( s\right) \cdot c,ss^{\prime
}).
\]
If we take $\phi ^{\ast }\left( C\right) =A,$ then $\phi ^{\ast
}\left( C,R,\partial \right) =(\phi ^{\ast }\left( C),S,\partial
^{\ast }\right) $ is a pullback crossed module. We now show this
as follows:

i) $\partial ^{\ast }:\phi ^{\ast }\left( C\right) \rightarrow S,$
$\partial ^{\ast }((c,s))=s$ is a crossed $S$-module. Since,

\[
\begin{array}{llll}
&  & \partial ^{\ast }\left( c,s\right) \cdot \left( c^{\prime
},s^{\prime
}\right) & =s\cdot \left( c^{\prime },s^{\prime }\right) \\
&  &  & =\left( \phi \left( s\right) \cdot c^{\prime },ss^{\prime
}\right)
\\
&  &  & =\left( \partial \left( c\right) \cdot c^{\prime
},ss^{\prime
}\right) \\
&  &  & =\left( cc^{\prime },ss^{\prime }\right) \\
&  &  & =\left( c,s\right) \left( c^{\prime },s^{\prime }\right)
\end{array}
\]
ii)
\[
\left( \phi ^{\prime },\phi \right) :(\phi ^{\ast }\left(
C),S,\partial ^{\ast }\right) \longrightarrow \left( C,R,\partial
\right)
\]
is a morphism of crossed module where $\phi ^{\prime }\left(
c,s\right) =c.$ Since
\[
\begin{array}{lll}
\phi ^{\prime }\left( s^{\prime }\cdot \left( c,s\right) \right) &
= & \phi ^{\prime }\left( \phi \left( s^{\prime }\right) \cdot
c,s^{\prime }s\right)
\\
& = & \phi \left( s^{\prime }\right) \cdot c \\
& = & \phi \left( s^{\prime }\right) \cdot \phi ^{\prime }\left(
c,s\right)
\end{array}
\]
and clearly $\partial \phi ^{\prime }=\phi \partial ^{\ast }.$

Suppose that
\[
\left( f,\phi \right) :\left( B,S,\mu \right) \longrightarrow
\left( C,R,\partial \right)
\]
is any crossed module morphism such that $\partial f=\phi \mu ,$
then there is a unique morphism
\[
\begin{array}{cccc}
f^{\ast }: & B & \longrightarrow & \phi ^{\ast }\left( C\right) \\
& x & \longmapsto & \left( f\left( x\right) ,\mu \left( x\right)
\right)
\end{array}
\]
since $\partial f\left( x\right) =\phi\mu \left( x\right) $ for all $%
x\in B .$\ Now, let us show that
$\left( f^{\ast },id_{S}\right)$ is a crossed $S$-module
morphism. For \ $x\in B  ,\ s\in S$%
\[
\begin{array}{lll}
f^{\ast }\left( s\cdot x\right) & = & \left( f\left( s\cdot
x\right) ,\mu
\left( s\cdot x\right) \right) \\
& = & \left( \phi \left( s\right) \cdot f\left( x\right) ,s\mu
\left(
x\right) \right) \\
& = & s\cdot \left( f\left( x\right) ,\mu \left( x\right) \right) \\
& = & s\cdot f^{\ast }\left( x\right) \\
& = & id_{S}(s)\cdot f^{\ast }\left( x\right),
\end{array}
\]
so $\left( f^{\ast },id_{S}\right) $ is a crossed $S$-module
morphism. \

Also, it is clear that  $\partial ^{\ast }f^{\ast }=\mu $ and
$\phi'f^{\ast }=f.$
Thus, we get a functor
\[
\phi ^{\ast }:{\bf XMod}/R\longrightarrow {\bf XMod}/S
\]
which gives our pullback crossed module.

This pullback crossed module can be given by a pullback diagram.

\begin{corollary}
Given a crossed module $\left( C,R,\partial \right) $ and a
morphism $\phi
:S\rightarrow R$ of $k$-algebras, there is a pullback diagram %
$$ \xymatrix@R=40pt@C=40pt{
  \phi^{*}(C) \ar[d]_{\partial^{\ast}} \ar[r]^{}
                & C \ar[d]^{\partial}  \\
  S  \ar[r]_{\phi}
                & R             } $$
\end{corollary}
\begin{proof}
It is straight forward from a direct calculation.
\end{proof}

We apply the result given in section \ref{s1} to fibred categories
for particular case.

We have a forgetful functor $p:\mathbf{XMod}\rightarrow $ $k$%
\textbf{-Alg }which sends $(C,R,\partial )\mapsto R.$ So, we get that:

\begin{proposition} The forgetful functor $p:\mathbf{XMod}\rightarrow $ $k$\textbf{-Alg}
is fibred and has a left adjoint.
\end{proposition}
\begin{proof}
The left adjoint of $p$ assigns to an algebra $S$ the crossed module $%
0\rightarrow S.$

Next, we give the pullback construction to prove that $p$ is fibred.
So let $\phi :S\rightarrow R$ be a morphism of algebras, and let
$\partial :C\rightarrow R$ be a crossed module. We define $\ A=\phi
^{\ast }\left( C\right) $ as in subsection \ref{s2}. So, $\partial
^{\ast }: A \rightarrow S$ is a crossed module and the morphism
$\phi ^{\prime}: A \rightarrow C $ given by $\phi ^{\prime }\left(
c,s\right) =c$ is cartesian.
\end{proof}

The above result in the cases of crossed modules of groups and groupoids appeared in
\cite{[bh]} and \cite{[bs]}, respectively.

\subsection{Examples of Pullback Crossed Modules}

{\bf 1.} Given crossed module $i=\partial :I\hookrightarrow R$ \
where $i$ is an inclusion of an ideal. The pullback crossed module
is
\[
\begin{array}{lllll}
\phi ^{\ast }\left( I,R,\partial \right) & = & \left( \phi ^{\ast
}\left(
I\right) ,S,\partial ^{\ast }\right) &  &  \\
& \cong & \left( \phi ^{-1}\left( I\right) ,S,\partial ^{\ast
}\right) &  &
\end{array}
\]
\ as,
\[
\begin{array}{lll}
\phi ^{\ast }\left( I\right) & = & \left\{ \left( i,s\right) \ |\
\phi
\left( s\right) =\partial \left( i\right) =i,\ s\in S,\ i\in I\right\} \\
& \cong & \left\{ s\in S\ |\ \phi \left( s\right) =i\in I\right\}
=\phi ^{-1}\left( I\right) \trianglelefteq S.
\end{array}
\]
The pullback diagram is
$$ \xymatrix@R=40pt@C=40pt{
  \phi^{-1}(I) \ar[d]_{\partial^{\ast}} \ar[r]^{}
                & I \ar[d]^{i}  \\
  S  \ar[r]_{\phi}
                & R             } $$
Particularly if $I=\left\{ 0\right\} ,$ then
\[
\phi ^{\ast }\left( \left\{ 0\right\} \right) \cong \left\{ s\in
S\ |\ \phi \left( s\right) =0\right\} =\text{Ker}\phi
\]
and so $\left( Ker\phi ,S,\partial ^{\ast }\right) $ is a pullback
crossed modules. Kernels are thus particular cases of
pullbacks.\\
Also if $ \phi$ is onto and $ I=R, $ then $ \phi ^{\ast
}(R)=R \times S $ $\bigskip $

{\bf 2.} Given a crossed module $0:M\longrightarrow R,$ $0(m)=0$ where $M\ $%
\ is any $R$-module, so it is also an $R$-algebra with zero
multiplication. Then
\[
\phi ^{\ast }\left( M,R,0\right) =\left( \phi ^{\ast }\left(
M\right) ,S,\partial ^{\ast }\right)
\]
where
\[
\begin{array}{lllll}
\phi ^{\ast }\left( M\right) & = & \left\{ \left( m,s\right) \in
M\times S\
|\ \phi \left( s\right) =\partial \left( m\right) =0\right\} &  &  \\
& = & \left\{ \left( m,s\right) \ |\ \phi \left( s\right) =0,\
s\in S\right\}
&  &  \\
& \cong & M\times Ker\phi &  &
\end{array}
\]
\ The corresponding pullback diagram is
$$ \xymatrix@R=40pt@C=40pt{
  M\times Ker\phi \ar[d]_{\partial^{\ast}} \ar[r]^{}
                & M \ar[d]^{0}  \\
  S  \ar[r]_{\phi}
                & R.             } $$
So if $\phi $ is injective $\left( \text{Ker}\phi =0\right) ,$ then
$M\cong \phi ^{\ast }\left( M\right) $. If $M=\left\{ 0\right\} ,$
then $\phi ^{\ast }\left( M\right) \cong $ Ker$\phi $.

{\bf 3.} If $\phi :S\rightarrow R$ is a morphism of algebras, then
there
may, or may not be a morphism $M(\phi ):M(S)\rightarrow M(R)$ such that
$$ \xymatrix@R=40pt@C=40pt{
  S \ar[d]_{} \ar[r]^{\phi}
                & R \ar[d]^{}  \\
  M(S)  \ar[r]_{M(\phi)}
                & M(R)            } $$
is a morphism of crossed modules. Using the following commutative
diagram
$$\xymatrix@R=40pt@C=40pt{  \phi^{\ast}(R) \ar[r]^{}\ar[d]_{{\partial}^{\ast}} & R\ar[d]_{} \\
S\ar[d]^{ }
 \ar[r]^{ } & M(R) \\ M(S)\ar[ur]_{M(\phi)}  & }$$
we get the pullback diagram
$$ \xymatrix@R=40pt@C=40pt{
  \phi^{\ast}(R) \ar[d]_{} \ar[r]^{}
                & R \ar[d]^{}  \\
  M(S)  \ar[r]_{M(\phi)}
                & M(R)            } $$
where
\[
\begin{array}{lll}
\phi ^{\ast }\left( R\right)  & = & \left\{ \left( \gamma
,r\right) :M(\phi )\left( \gamma \right) =\partial \left( r\right)
,\ \gamma \in M(S),\ r\in R\right\} .
\end{array}
\]

\section{Induced Crossed Modules}

We will consider a functor $\phi _{\ast }:{\bf XMod}/S\longrightarrow {\bf %
XMod}/R$ left adjoint defined to the pullback $\phi ^{\ast }$ of
the previous section. This functor has already been defined by
Porter \cite{[p1]} for which he call it ``extension along a
morphism''. But we defined this functor by the universal property
and analysed this construction deeply.

\begin{definition}\label{d1}
For any crossed $S$-module $\partial :D\rightarrow S$ and
$k$-algebra morphism $\phi :S\rightarrow R,$ the {\it induced
crossed module} can be given by

i) a crossed $R$-module $\phi _{\ast }\left( D,S,\partial \right)
=\left(
\partial _{\ast }:\phi _{\ast }(D)\rightarrow R\right) $

ii) Given
\[
\left( f,\phi \right) :\left( D,S,\partial \right) \longrightarrow
\left( B,R,\eta \right)
\]
crossed module morphism, then there is a unique $\left( f_{\ast
},id_{R}\right) $ crossed $R$-module morphism such that commutes
the following diagram
$$
\xymatrix@R=40pt@C=40pt{
  (D,S,\partial) \ar[d]_{(f,\phi)} \ar[dr]^{(\phi',\phi)}   \\
  (B,R,\eta) \ar@{<.}[r]_{(f_{\ast},id_{R})}  & (\phi_{\ast}(D),R,\partial_{\ast})              }
$$
or more simply as
$$\xymatrix@R=20pt@C=20pt{
  D \ar[dd]_{\partial} \ar[rr]^{f}\ar[dr]_{\phi'}
              &  & B \ar[dd]^{\eta}  \\
&\phi_{\ast}(D)\ar@{.>}[ur]|-{f_{\ast}}\ar[dr]_{\partial_{\ast}} &
 \\ S  \ar[rr]_{\phi}
               & & R             }$$
\end{definition}

\subsection{Construction of Induced Crossed Module}\label{s4}

We will construct the induced crossed module as follows. Given a $k$%
-algebra morphism $\phi :S\rightarrow R$ and a crossed module
$\partial :D\longrightarrow S,$ and let the set
\[
F(D\times R)
\]
be a free algebra generated by the elements of $D\times R$. Let
$P$ be the ideal generated by all the relations of the three
following types:
\begin{eqnarray*}
&&(d_{1},r)+(d_{2},r)=(d_{1}+d_{2},r)\\
&&(s\cdot d,r)=(d,\phi (s)r)\\
&&(d_{1},r_{1})(d_{2},r_{2})=(d_{2},r_{1}(\phi \partial d_{1})r_{2})
\end{eqnarray*}
for any $d,d_{1},d_{2}\in D,\ $and $r\in R,\ s\in S$

We define
\[
D\otimes _{S}R=F(D\times R)/P.
\]
This is an $R$-algebra with
\[
r^{\prime }\cdot (d\otimes r)=d\otimes r^{\prime }r
\]
for $d\in D,\ r,\ r^{\prime }\in R.$ If we take $\phi _{\ast
}(D)=D\otimes _{S}R$, then
\[
\phi _{\ast }(D,S,\partial )=\left( \phi _{\ast }(D),R,\partial
_{\ast }\right)
\]
is a induced crossed module. We will see it as follows:

i)
\[
\begin{array}{cccc}
\partial _{\ast }: & D\otimes _{S}R & \longrightarrow & R \\
& d\otimes r & \longmapsto & \phi \partial \left( d\right) r
\end{array}
\]
\begin{eqnarray*}
{}\partial _{\ast }({d\otimes r})\cdot {(d}_{1}\otimes r_{1}{)}
&=&\left(
\left( \phi \partial d\right) r\right) \cdot (d_{1}\otimes r_{1}) \\
&=&(d_{1}\otimes \phi (\partial d)rr_{1}) \\
&=&(\partial d\cdot d_{1}\otimes rr_{1}) \\
&=&(dd_{1}\otimes rr_{1}) \\
&=&(d\otimes r)(d_{1}\otimes r_{1})
\end{eqnarray*}
so $\partial _{\ast }$ is a crossed $R$-module.

ii)Since $R$ has a\ unit, $\phi ^{\prime }:D\rightarrow \phi _{\ast
}\left( D\right) $ is defined by $\phi ^{\prime }(d)=(d\otimes 1),$
then
\[
\begin{array}{lllll}
\phi ^{\prime }(s\cdot d) & = & (s\cdot d\otimes 1) &  &  \\
& = & (d\otimes \phi \left( s\right) ) &  &  \\
& = & \phi \left( s\right) \cdot (d\otimes 1) &  &  \\
& = & \phi (s)\cdot \phi ^{\prime }(d) &  &
\end{array}
\]
for $d\in D,$ $s\in S.$ So $(\phi ^{\prime },\phi ):(D,S,\partial
)\rightarrow (\phi _{\ast }\left( D\right) ,R,\partial _{\ast })$
is a crossed module morphism.

Let
\[
\left( f,\phi \right) :\left( D,S,\partial \right) \longrightarrow
\left( B,R,\eta \right)
\]
be any crossed module morphism. Then there is a morphism $f_{\ast
}$ given by
\[
\begin{array}{cccc}
f_{\ast }: & \phi _{\ast }\left( D\right) & \longrightarrow & B \\
& d\otimes r & \longmapsto & r\cdot f(d)
\end{array}
\]
\ Also, for\ $x\in \phi _{\ast }\left( D\right) ,\ r\in R$
\[
\begin{array}{llll}
f_{\ast }\left( r\cdot (d_{1}\otimes r_{1})\right) & = & f_{\ast
}(d_{1}\otimes rr_{1}) &  \\
& = & rr_{1}\cdot f(d_{1}) &  \\
& = & r\cdot (r_{1}\cdot f(d_{1})) &  \\
& = & r\cdot f_{\ast }(d_{1}\otimes r_{1}) &  \\
& = & id_{R}(r)\cdot f_{\ast }(d_{1}\otimes r_{1}) &
\end{array}
\]
so $\left( f_{\ast },id_{R}\right) $ is a crossed $R$-module
morphism. \

Finally,
\[
\begin{array}{lll}
\left( \eta f_{\ast }\right) \left( (d\otimes r)\right) & = & \eta
\left(
f_{\ast }(d\otimes r)\right) \\
& = & \eta \left( r\cdot f(d)\right) \\
& = & r\eta \left( f(d)\right) \\
& = & r\phi \left( \partial (d)\right) \\
& = & \partial _{\ast }\left( d\otimes r\right) \\
& = & id_{R}\partial _{\ast }(d\otimes r)
\end{array}
\]
for each $(d\otimes r)\in \phi _{\ast }\left( D\right) $ and also
$f_{\ast }\phi ^{\prime }=f$. Thus, we get a functor
\[
\phi _{\ast }:{\bf XMod}/S\longrightarrow {\bf XMod}/R
\]
which gives our induced crossed module.

The induced crossed module can be explain in terms of pushout
diagram.

\begin{corollary}
Let $\partial :D\rightarrow S$ be a crossed $S$-module and $\phi
:S\rightarrow R,$ $k$-algebra morphism. Then there is an induced diagram %

$$\xymatrix@R=40pt@C=40pt{
  D \ar[d]_{\partial} \ar[r]^{\phi '}
                & \phi_{\ast}(D) \ar[d]^{\partial_{\ast}} \\
  S \ar[r]_{\phi}
                & R.             }$$
$\Box $
\end{corollary}

If we consider the forgetful functor  $p:\mathbf{XMod}\rightarrow $
$k$\textbf{-Alg} mentioned in section \ref{s3}, then we get
following result.
\begin{proposition}
The forgetful functor $p:\mathbf{XMod}\rightarrow $ $k$\textbf{-Alg}
is cofibred.
\end{proposition}
\begin{proof}
let $\phi :S\rightarrow R$ be a morphism of algebras, and let
$\partial :D\rightarrow S$ be a crossed module. The construction of
$\phi _{\ast }\left( D\right)$ and of  $\partial _{\ast }: \phi
_{\ast }\left( D\right) \rightarrow R$ is as in subsection \ref{s4}.
We get a crossed module morphism $(\phi ^{\prime },\phi
):(D,S,\partial)\rightarrow (\phi _{\ast }\left( D\right)
,R,\partial _{\ast })$ which is cocartesian.
\end{proof}
\begin{remark}
\cite{[s],[p1]} For any $k$-algebra morphism $\phi :S\rightarrow R,$
there is an adjoint functor pair $\left( \phi ^{\ast },\phi _{\ast
}\right).$ Thus the category $\mathbf{XMod}$ is bifibred over
$k$\textbf{-Alg} by the forgetful functor $
p:\mathbf{XMod}\rightarrow $ $k$\textbf{-Alg}.
\end{remark}

\subsection{Examples of Induced Crossed Modules}

{\bf 1.} $\ $Let $D=S$ and $\partial =id_{S}:S\rightarrow S$ \ be
identity crossed $S$-modules. \ The induced crossed module diagram is \ %
$$\xymatrix@R=40pt@C=40pt{
  S \ar[d]_{i=\partial} \ar[r]^{\psi}
                & \phi _{\ast }(S)\dto^{\partial_\ast} \\
  S \ar[r]_{\phi}
                & R             }$$
where $\phi_{\ast }\left( S\right)=S\otimes _{S}R.$

(Remark: $S$ has not unit, otherwise $S\otimes _{S}R\cong R)$. When we take $%
S=k^{+}[X],$ the positively graded part of the polynomial algebra
over $k$ on the set of generators, $X$ we have the induced crossed
module constructed in subsection \ref{s4}. $ \partial _{\ast
}:k^{+}[X]\otimes _{k^{+}[X]}R\rightarrow R$ \ which is the free
$R$-module on $\ f:X\longrightarrow R.$ Thus the free crossed
modules is the special case of the induced crossed modules. We will
examine this with respect to subsection \ref{s5}.

Considering the free crossed module construction given in the
first section, we have a diagram
$$\xymatrix@R=40pt@C=40pt{
  k^{+}[X]\times R \ar[d]_{} \ar[r]^{\theta} &   R     \\
  \left( k^{+}[X]\times R \right)/P=k^{+}[X]\otimes _{k^{+}[X]}R \ar[ur]_{\partial_{\ast}}                     }$$
with $ \theta(p,r)=\partial_{\ast}(p\otimes r)=\phi(p)r $ for all
$p\in k^{+}[X],\,r\in R,$  where $P$ is an ideal generated by all
the relations given in subsection \ref{s4} From
\[
\begin{array}{ccl}
\theta \left( \left( p_{1},r\right) +\left( p_{2},r\right) -\left(
p_{1}+p_{2},r\right) \right) & = & \theta \left( p_{1},r\right)
+\theta
\left( p_{2},r\right) -\theta \left( p_{1}+p_{2},r\right) \\
& = & \phi \left( p_{1}\right) r+\phi \left( p_{2}\right) r-\left(
\phi
\left( p_{1}\right) +\phi \left( p_{2}\right) \right) r \\
& = & 0
\end{array}
\]
\[
\begin{array}{ccl}
\theta \left( \left( p\cdot q,r\right) -\left( q,\phi \left(
p\right) r\right) \right) & = & \theta \left( pq,r\right) -\theta
\left( q,\phi \left(
p\right) r\right) \\
& = & \phi \left( pq\right) r+\phi \left( q\right) \phi \left( p\right) r \\
& = & 0
\end{array}
\]
\[
\begin{array}{ccl}
\theta \left( \left( p_{1},r_{1}\right) \left( p_{2},r_{2}\right)
-\left( p_{2},r_{1}\phi \partial p_{1}r_{2}\right) \right) & = &
\theta \left( p_{1},r_{1}\right) \theta \left( p_{2},r_{2}\right)
-\theta \left(
p_{2},r_{1}\phi \partial p_{1}r_{2}\right) \\
& = & \phi \left( p_{1}\right) r_{1}\phi \left( p_{2}\right)
r_{2}-\phi
\left( p_{2}\right) r_{1}\phi \partial \left( p_{1}\right) r_{2} \\
& = & 0,
\end{array}
\]
we get $\theta (P)=0,$ and have the pushout diagram
$$
\xymatrix@R=40pt@C=40pt{&&&& A\ar[ddl]^{\delta}\\
  &X \ar@{=}[d]_{id} \ar[r]^{} & k^{+}[X]\ar[urr]^{} \ar[d]_{\partial}
  \ar[r]^{\phi'} &\phi _{\ast }\left( k^{+}[X]\right)
   \ar[d]^{\partial_{\ast}}\ar@{.>}[ur]|-{} \\
  &X \ar[r]_{i} & k^{+}[X] \ar[r]_{\phi} & R   }
$$
where $ \phi _{\ast }\left( k^{+}[X]\right) =k^{+}[X]\otimes
_{k^{+}[X]}R.$

{\bf 2.} \ Let $D\ $be $S$-module and $0=\partial :D\rightarrow S$
be zero morphism. \ The pushout diagram is\
$$\xymatrix@R=40pt@C=40pt{
  D \ar[d]_{\partial=0} \ar[r]^{\psi}
                & \phi _{\ast }(D) \ar[d]^{\partial_\ast}  \\
  S  \ar[r]_{\phi}
                & R             }
$$
where
\[
\begin{array}{lllll}
\partial _{\ast }\left( d\otimes r\right) & = & \phi \left( \partial \left(
d\right) \right) r &  &  \\
& = & \phi \left( 0\right) r &  &  \\
& = & 0r=0 &  &
\end{array}
\]
so $\partial _{\ast }=0$ and $P=0.$ Thus,
\[
\phi _{\ast }\left( D\right) =F\left( D\times R\right)
\]
Then, the induced crossed modules is a free $S$-module on $D\times
R.$

{\bf 3.} Given crossed module $i=\partial :I\hookrightarrow S$
where $i$ inclusion of an ideal. Using any surjective homomorphism
$\phi :S\rightarrow S/I$ \ the induced diagram is
$$ \xymatrix@R=40pt@C=40pt{
  I \ar[d]_{i=\partial} \ar[r]^{\psi}
                & \phi _{\ast }(I) \ar[d]^{\partial_\ast}  \\
  S \ar[r]_{\phi}
                & S/I             }$$
Thus we get$\ \phi _{\ast }\left( I\right) = I\otimes \left(
S/I\right)\cong I/I^{2}$ which is an $S/I$-module. So $\phi _{\ast
}$ does not preserve ideals.

{\bf 4. }Given crossed module $\partial :S\rightarrow M(S)$ with
$\partial
(s)=\delta _{s}$ and $\delta_{s}(s^{\prime })=ss^{\prime }$ for all $%
s,s^{\prime }\in S$. If $M(\phi ):M(S)\rightarrow M(R)$ is a morphism where $%
\phi :S\rightarrow R$ is a morphism of algebras such that $\left(
\phi ,M(\phi )\right) $ is a morphism of crossed modules. We get
the induced diagram
$$ \xymatrix@R=40pt@C=40pt{
  S \ar[d]_{\partial} \ar[r]^{}
                & \phi_{\ast}(S) \ar[d]^{\partial _{\ast}}  \\
  M(S)  \ar[r]_{M(\phi)}
                & M(R)            } $$
where $\phi _{\ast }\left( S\right) =S\otimes M(R). $
\subsection{Properties of Induced Crossed Module}

$\phi _{\ast }\left( D\right) $ induced crossed $R$-module can be
expressed more simply for the case when $\phi :S\rightarrow R,\
k$-algebra morphism, is an epimorphism or monomorphism.

\subsubsection{Epimorphism case:}

\begin{proposition}
\cite{[s]} Let $\partial :D\rightarrow S$ be a crossed $S$-module
and $\phi :S\rightarrow R$ epimorphism with Ker$\phi =K$. Then
\[
\phi _{\ast }\left( D\right) \cong D/KD
\]
where $KD$ is an ideal of $D$ generated by $\{k\cdot d\mid d\in
D,\,k\in K\}.$
\end{proposition}

\begin{proof}
Because $S$ acts on $D/KD$, $K$ acts trivially on $D/KD$ and $\phi$ is an epimorphism, $R\cong S/K$ acts on
$D/KD$. Indeed, because of
\[
\begin{array}{rcl}
S\times D/KD & \longrightarrow & D/KD \\
\left( s,d+KD\right) & \longmapsto & s\cdot \left( d+KD\right)
=sd+KD
\end{array}
\]
and
\[
\begin{array}{rcl}
K\times D/KD & \longrightarrow & D/KD \\
\left( k,d+KD\right) & \longmapsto & k\cdot \left( d+KD\right)
=kd+KD
\end{array}
\]
$S/K$ acts on $D/KD$ as following
$$\xymatrix@R=40pt@C=40pt{
  S\times D/KD \ar[d]_{} \ar[r]^{}
                & D/KD \ar@{=}[d]^{}  \\
  S/K\times D/KD\ar[r]_{}
                & D/KD             }$$
                $$\xymatrix@R=40pt@C=40pt{
  ( s,d+KD) \ar@{|->}[d]_{} \ar@{|->}[r]^{}
                & sd+KD \ar@{|->}[d]^{}  \\
  ( s+K,d+KD) \ar@{|->}[r]
                & sd+KD             }$$

$\beta :D/KD\rightarrow R$ given by $\beta \left( d+KD\right)
=\partial (d)+K $ is a crossed $R$-module.\ Indeed,
\[
\begin{array}{llll}
& \beta \left( d+KD\right) \cdot \left( d^{\prime }+KD\right) & =
& \left(
\partial d+K\right) \cdot \left( d^{\prime }+KD\right) \\
&  & = & \partial \left( d\right) \cdot d^{\prime }+KD \\
&  & = & dd^{\prime }+KD \\
&  & = & \left( d+KD\right) \left( d^{\prime }+KD\right) \\
&  &  &
\end{array}
\]
$\left( \rho ,\phi \right) :\left( D,S,\partial \right)
\longrightarrow \left( D/KD,R,\beta \right) $ is a crossed module
morphism where $\rho :D\longrightarrow D/KD,\ \  \rho (d)=d+KD$ since
$\rho \left( s\cdot d\right) =\phi \left( s\right) \cdot \rho
\left( d\right) $.

Suppose that the following diagram of crossed module is commutative.
$$ \xymatrix@R=40pt@C=40pt{
  D \ar[d]_{\partial} \ar[r]^{\rho '}
                & D' \ar[d]^{\beta '}  \\
  S  \ar[r]_{\phi}
                & R             }
$$
Since $\rho ^{\prime }(s\cdot d)=\phi (s)\cdot \rho ^{\prime }(d)$
for any $d\in D,\,s\in S,\,$we have
\[
\rho ^{\prime }\left( k\cdot d\right) =\phi \left( k\right) \cdot
\rho ^{\prime }\left( d\right) =0\cdot \rho ^{\prime }\left(
d\right) =0
\]
so $\rho ^{\prime }\left( KD\right) =0$. Then, there is a unique morphism $%
\mu :\left( D/KD\right) \rightarrow D ^{\prime }$ given by $\mu
\left( d+KD\right) =\rho ^{\prime }\left( d\right) $ such that
$\mu \rho =\rho ^{\prime }$\ and $\mu $ is well defined, because
of $\rho ^{\prime }\left( KD\right) =0$. Finally, the diagram
$$
\xymatrix@R=40pt@C=40pt{  &D \ar[d]_{\partial} \ar[r]^{\rho} & D/KD
\ar[d]_{\beta} \ar[r]^{\mu} & D'
   \ar[dl]^{\beta '} \\
  &S \ar[r]_{\phi} & R  &    }
$$
commutes, since for all $d\in D$
\[
\begin{array}{lllll}
\beta \left( d+KD\right) & = & \phi \partial d &  &  \\
& = & \partial d+K &  &  \\
& = & \beta ^{\prime }\rho ^{\prime }\left( d\right) &  &  \\
& = & \beta ^{\prime }\mu \left( d+KD\right) &  &
\end{array}
\]
and
\[
\mu (r\cdot (d+KD))=\mu ((s\cdot d)+KD)=\mu (\rho (s\cdot d))=\rho
^{\prime }\left( s\cdot d\right)\] \[ \ \ \ \ \ \ \ \ \ \ \  \ \ \ \
\ \ \ \ \ \ =\phi \left( s\right) \cdot \rho ^{\prime }\left(
d\right) =r\cdot \mu \rho (d)=r\cdot \mu (d+KD)
\]
so $\mu $ preserves the actions.
\end{proof}

\subsubsection{Monomorphism case:}

In this subsection we consider the crossed modules induced by a
morphism $\phi :S\rightarrow R$ of $k$-algebras, the particular
case when $S$ is an ideal of $R.$

If $d\in D$, then the class of $d$ in $D/D^{2}$ is written as $\left[ d%
\right] $. Then the augmentation ideal of $I(R/S)$
of a quotient algebra $R/S$
has the basis \{\={e}$_{i_{1}}$\={e}$_{i_{2}}\ldots $\={e}$%
_{i_{p}},i_{1}\leq i_{2}\leq \cdots i_{p},i_{j}\in I$\}$_{(i)\neq
\emptyset
} $, where \={e}$_{i_{j}}$ is the projection of \ the basic element e$%
_{i_{j}}\in I(R)$ on $R/S.$

\begin{theorem}
Let $D$ $\subseteq S$ be ideals of $R$ so that $R$ acts on $S$ and
$D$ by multiplication. Let $\partial :D\rightarrow S$ , $\phi
:S\rightarrow R$ be the inclusions and let ${\cal D}$ denote the
crossed module $(D,S,\partial )$ with the multiplication action.
Then the induced crossed $R$-module $\phi _{\ast }\left( {\cal
D}\right) $ is isomorphic as a crossed $R$-module to
\[
\begin{array}{cccc}
\zeta : & D\times \left( D/D^{2}\otimes I\left( R/S\right) \right)
&
\longrightarrow & R \\
& \left( d,\left[ t\right] \otimes \bar{x}\right) & \longmapsto &
d.
\end{array}
\]
The action is given by
\[
r\cdot (d,\left[ t\right] \otimes \overline{x})=(r\cdot d,\left[
d\right]
\otimes \overline{r}+\left[ t\right] \otimes \overline{r}x-\left[ x\cdot t%
\right] \otimes \overline{r})
\]
for $d,\ t\in D$ ; $\bar{x}\in I(R/S)$ where $\overline{r}$,
$\overline{x}$ denote the image of $r$, $x$ in $R/S$,
respectively.
\end{theorem}

\begin{proof}
First we will show that
\[
{\cal T}=(\zeta :T=\left( D\times (D/D^{2}\otimes I(R/S))\right)
\rightarrow R),\,\zeta \left( d,\left[ t\right] \otimes
\bar{x}\right) =d
\]
is a crossed module with the given action:
\[
\begin{array}{llll}
& \zeta \left( d^{\prime },\left[ t^{\prime }\right] \otimes
\bar{x}^{\prime }\right) \cdot \left( d,\left[ t\right] \otimes
\bar{x}\right) & = &
d^{\prime }\cdot \left( d,\left[ t\right] \otimes \bar{x}\right) \\
&  & = & \left( d^{\prime }\cdot d,\left[ d\right] \otimes \overline{%
d^{\prime }}+\left[ t\right] \otimes \overline{d^{\prime }}x-\left[ x\cdot t%
\right] \otimes \overline{d^{\prime }}\right) \\
&  & = & \left( d^{\prime }d,0\right) \\
&  & = & \left( d^{\prime }d,\left[ t^{\prime }t\right] \otimes \bar{x}%
^{\prime }\bar{x}\right) \\
&  & = & \left( d^{\prime },\left[ t^{\prime }\right] \otimes \bar{x}%
^{\prime }\right) \left( d,\left[ t\right] \otimes \bar{x}\right)
\end{array}
\]
Consider $i:D\rightarrow D\times \left( D/D^{2}\otimes I\left(
R/S\right) \right) ,\,i(d)=\left( d,0\right) .$ We have the
following diagram
$$
\xymatrix@R=40pt@C=40pt{&&&& C\ar[ddl]^{\alpha}\\
  & & D\ar[urr]^{\beta} \ar[d]_{\partial} \ar[r]^{i} &   T
   \ar[d]^{\zeta}\ar@{.>}[ur]|-{\widetilde{\phi}} \\
  && S \ar[r]_{\phi} & R.   }
$$
Clearly we have a morphism of crossed modules $\left( i,\phi \right)
:{\cal D}\rightarrow {\cal T}$. \ We just verify that this morphism
satisfies condition ii) of Definition \ref{d1}. That is, when a
morphism of crossed module $(\beta ,\phi ):(D,S,\partial
)\rightarrow (C,R,\alpha )$ is given we prove that there is a unique
morphism $\widetilde{\phi } :T=D\times (D/D^{2}\otimes
I(R/S))\rightarrow C$ such that $\widetilde{\phi } i=\beta $ and
$\alpha \widetilde{\phi }=\zeta $. Since $\widetilde{\phi }$ has to
be a homomorphism and preserve the action we have
\[
\begin{array}{llll}
\widetilde{\phi }\left( d,\left[ t\right] \otimes
\bar{e}_{(i)}\right) & = &
\widetilde{\phi }\left( (d,0)+\left( 0,\left[ t\right] \otimes \overline{e}%
_{(i)}\right) \right) &  \\
& = & \widetilde{\phi }\left( (d,0)+(e_{(i)}\cdot t,\left[
t\right] \otimes
\overline{e}_{(i)})+(-e_{(i)}\cdot t,0)\right) &  \\
& = & \widetilde{\phi }((d,0)+\left( e_{(i)}\cdot
(t,0)+(-e_{(i)}\cdot
t,0)\right) &  \\
& = & \widetilde{\phi }((d,0))+\widetilde{\phi }\left( e_{(i)}\cdot (t,0)+(-e_{(i)}\cdot t,0)\right) &  \\
& = & \widetilde{\phi }((d,0))+\widetilde{\phi }( e_{(i)}\cdot
(t,0))+\widetilde{\phi }(-e_{(i)}\cdot t,0) &  \\
& = & \widetilde{\phi }i(d)+e_{(i)}\cdot \widetilde{\phi }i(t)-\widetilde{%
\phi }i(e_{(i)}\cdot t) &  \\
& = & \beta (d)+e_{(i)}\cdot \beta (t)-\beta (e_{(i)}\cdot t)
\end{array}
\]
for any $d\in D,\left( \left[ t\right] \otimes
\bar{e}_{(i)}\right) \in D/D^{2}\otimes I\left( R/S\right)$. This
proves uniqueness of any such a $\widetilde{\phi }.$ We now prove
that this formula gives a well-defined morphism.

It is immediate from the formula that $\widetilde{\phi }:D\times
(D/D^{2}\otimes I(R/S))\rightarrow C$ \ has to be $\beta $ on the
first factor and is defined on the second one by the map $$\left[
t\right] \otimes \bar{e}_{(i)}\longmapsto e_{(i)}\cdot \beta
(t)-\beta (e_{(i)}\cdot t).$$ We have to check that this latter map
is well defined homomorphism.

We define the function
\[
\begin{array}{llll}
\gamma _{r}: & D & \longrightarrow & C \\
& d & \longmapsto & r\cdot \beta (d)-\beta \left( r\cdot d\right)
\end{array}
\]
and prove in turn the following statements.

{\bf 3.5} $\gamma _{r}(D)$ is contained in the annihilator
Ann($C$) of$\ C$.

{\bf Proof of 3.5 }We use the fact that if $d\in D$, then
\[
\begin{array}{llll}
\alpha \gamma _{r}(d) & = & \alpha \left( r\cdot \beta (d)-\beta
\left(
r\cdot d\right) \right) &  \\
& = & \alpha \left( r\cdot \beta (d))-\alpha (\beta \left( r\cdot
d\right)
\right) &  \\
& = & r\alpha \left( \beta (d))-\phi \partial \left( r\cdot d\right)
\right)\ \ \ \ \ (\text{since}\  \alpha\beta=\phi \partial)
&  \\
& = & r\phi \partial (d)-\phi \left( r\partial \left( d\right) \right) &  \\
& = & r\partial (d)-r\partial \left( d\right) =0 &
\end{array}
\]
and that $\left( C,R,\alpha \right) $ is a crossed module. So
$\gamma _{r}(D)\subseteq $Ker$\alpha =$Ann$(C)$

{\bf 3.6 }$\gamma _{r}$ is a morphism which factors through
$D/D^{2}$.

{\bf Proof of 3.6 } Let $d,d^{\prime }\in D$ be, then
\[
\begin{array}{lll}
\gamma _{r}(dd^{\prime }) & =r\cdot \beta \left( dd^{\prime
}\right) -\beta
\left( r\cdot \left( dd^{\prime }\right) \right) &  \\
& =r\cdot \left( \beta \left( d\right) \beta \left( d^{\prime
}\right)
\right) -\beta \left( r\cdot d\right) \beta \left( d^{\prime }\right) &  \\
& =r\cdot \beta \left( d\right) \beta \left( d^{\prime }\right)
-\beta
\left( r\cdot d\right) \beta \left( d^{\prime }\right) &  \\
& =\gamma _{r}(d)\beta (d^{\prime }) &  \\
& =0\  \qquad\qquad\qquad\qquad (\text{since}\ \gamma _{r}(d) \ \text{belongs to}\ Ann(C))&\\
& =\gamma _{r}(d)\gamma _{r}(d^{\prime }) &
\end{array}
\]
Consequently $\gamma _{r}$ is a homomorphism of commutative
algebras that factors through $D/D^{2}.$

{\bf 3.7 } We define a morphism
\[
\begin{array}{rccc}
\gamma : & D/D^{2}\otimes I(R/S) & \longrightarrow & C \\
& \left[ d\right] \otimes \overline{e}_{(i)} & \longmapsto &
\gamma _{e_{(i)}}(d)
\end{array}
\]

{\bf Proof of 3.7 \ }Since $\beta $ is a $S$-equivariant
homomorphism, the morphisms $\gamma _{r}$ depend only on the classes
$\overline{r}$ of $r$ in $R/S$. Thus, we define the morphism $\gamma
$ as mentioned.

{\bf 3.8 }The function $\widetilde{\phi }$ defined in the theorem satisfies $%
\widetilde{\phi }i=\beta $ and is a well defined morphism of
$R$-crossed modules.

{\bf Proof of 3.8 }The function $\widetilde{\phi }$ is clearly a
well-defined morphism of commutative algebras, since it is of the
form $ \widetilde{\phi }(d,u)=\beta (d)+\gamma(u) $, where $\beta $
and $\gamma $ are the homomorphisms of commutative algebras and
$\gamma \left( u\right) $ belongs to the annihilator of $C$.
\[
\begin{array}{ll}
\widetilde{\phi }\left( \left( d,u\right) \left( d^{\prime
},u^{\prime }\right) \right) & =\widetilde{\phi }\left( dd^{\prime
},uu^{\prime }\right)
\\
& =\beta (dd^{\prime })+\gamma \left( uu^{\prime }\right) \\
& =\beta (d)\beta (d^{\prime })+\gamma \left( u\right) \gamma
\left(
u^{\prime }\right) \\
& =\beta (d)\beta (d^{\prime })+\beta (d)\gamma \left( u^{\prime
}\right) +\gamma \left( u\right) \beta (d^{\prime })+\gamma \left(
u\right) \gamma
\left( u^{\prime }\right) \ (\text{since}\ \gamma(u),\ \gamma(u')\in Ann(C))\\
& =\left( \beta (d)+\gamma \left( u\right) \right) \left( \beta
(d^{\prime
})+\gamma \left( u^{\prime }\right) \right) \\
& =\widetilde{\phi }\left( d,u\right) \widetilde{\phi }\left(
d^{\prime },u^{\prime }\right)
\end{array}
\]
Further, $\widetilde{\phi }i=\beta $ and $\alpha \widetilde{\phi
}=\zeta $, as $\alpha \gamma $ is trivial:
\[
\begin{array}{lll}
\alpha \gamma \left( \left[ t\right] \otimes \bar{e}_{(i)}\right)
& = & \alpha \left( e_{(i)}\cdot \beta \left( t\right) -\beta
\left( e_{(i)}\cdot
t\right) \right) \\
& \stackrel{}{=} & e_{(i)}\cdot \alpha \left( \beta \left(
t\right) \right)
-\alpha \left( \beta \left( e_{(i)}\cdot t\right) \right) \\
& = & e_{(i)}\cdot \phi \partial \left( t\right) -\phi \partial
\left(
e_{(i)}\cdot t\right) \ \ \ \ \ (\text{since}\  \alpha\beta=\phi \partial)\\
& = & e_{(i)}\cdot \partial \left( t\right) -\partial \left(
e_{(i)}\cdot
t\right) \\
& = & 0.
\end{array}
\]
Finally, we prove that $\widetilde{\phi }$ preserves the action.
This is the crucial part of the argument. Let $d,\ t\in D$,$r\in R$ and $\overline{e}%
_{(i)}$ be an element in the basis of $I(R/S)$, then
\[
\begin{array}{lll}
\widetilde{\phi }\left( r\cdot \left( d,\left[ t\right] \otimes \overline{e}%
_{(i)}\right) \right) & = & \widetilde{\phi }\left( r\cdot
d,\left[ d\right] \otimes \overline{r}+\left[ t\right] \otimes
\overline{r}e_{(i)}-\left[
e_{(i)}\cdot t\right] \otimes \overline{r}\right) \\
& = & \beta (r\cdot d)+\gamma \left( \left[ d\right] \otimes \overline{r}+%
\left[ t\right] \otimes \overline{r}e_{(i)}-\left[ e_{(i)}\cdot
t\right]
\otimes \overline{r}\right) \ \ \ (\text{since}\ \widetilde{\phi }(d,u)=\beta (d)+\gamma(u) )\\
& = & \beta \left( r\cdot d\right) +\gamma _{r}\left( d\right)
+\gamma
_{re_{(i)}}\left( t\right) -\gamma _{r}\left( e_{(i)}\cdot t\right) \ \ \ \ \ \ \ \ \ (\text{definition of}\ \gamma )\\
& = & \beta \left( r\cdot d\right) +r\cdot \beta \left( d\right)
-\beta \left( r\cdot d\right) +re_{(i)}\cdot \beta \left( t\right)
-\beta \left(
re_{(i)}\cdot t\right)  \ \ \ \ \ (\text{definition of}\ \gamma_r )\\
&  & -r\cdot \beta \left( e_{(i)}\cdot t\right) +\beta \left( r\cdot
\left(
e_{(i)}\cdot t\right) \right) \\
& = & \beta \left( r\cdot d\right) +re_{(i)}\cdot \beta \left(
t\right)
-r\cdot \beta \left( e_{(i)}\cdot t\right) \\
& = & r\cdot \left( \beta \left( d\right) +e_{(i)}\cdot \beta
\left(
t\right) -\beta \left( e_{(i)}\cdot t\right) \right) \\
& = & r\cdot \left( \beta \left( d\right) +\gamma _{e_{(i)}}\left(
t\right)
\right) \ \ \ \  \ \ \ \ \ (\text{definition of}\   \widetilde{\phi }) \\
& = & r\cdot \widetilde{\phi }\left( d,\left[ t\right] \otimes \overline{e}%
_{(i)}\right).
\end{array}
\]
\end{proof}

\textbf{Note:} As the free crossed modules are the special case of
the induced crossed modules we can discuss the free crossed module
$C\rightarrow R$ of commutative algebras which was shown in
\cite{[p2]} to have $C\cong R^{n}/d(\Lambda ^{2}R^{n})$, i.e. the
$2$nd Koszul complex term module the $2$-boundaries where $d:\Lambda
^{2}R^{n}\rightarrow R^{n}$ the Koszul differential. The idea is
that $X$ is a finite set and given a function $f:X\rightarrow R,$
then one can get
\[
C=R^{+}[X]/P
\]
constructed in section \ref{s5} where $P$ is generated by the
elements $\ pq-\theta (p)q$ for $p,q\in R^{+}[X].$ When we take
$S=k^{+}[X]$ in a morphism $\phi :S\rightarrow R,$ there is the
induced crossed module
\[
\partial :k^{+}[X]\otimes _{k^{+}[X]}R\rightarrow R
\]
and restate
\[
k^{+}[X]\otimes _{k^{+}[X]}R\cong R^{n}/d\left( \Lambda
^{2}R^{n}\right)
\]
This connection with the Koszul complex is taken in
\cite{[p1],[p2]}. (see also \cite{[ap]})

Ummahan Ege Arslan \\
uege@ogu.edu.tr\\
Eski\c sehir Osmangazi University, Department of Mathematics and Computer Sciences, \\
26480, Eski\c sehir, Turkey\\

\noindent\"{O}zg\"un G\"urmen \\
ogurmen@dpu.edu.tr\\
Dumlup\i nar University, Department of Mathematics,
43270, K\"{u}tahya, Turkey

\end{document}